\documentclass[12pt]{article}
\usepackage[latin1]{inputenc}
\usepackage{amssymb,amsmath,amsfonts}
\usepackage[notcite,notref]{showkeys}
\usepackage{xcolor}
\newtheorem{theorem}{Theorem}
\newtheorem{proposition}[theorem]{Proposition}
\newtheorem{lemma}[theorem]{Lemma}

\newtheorem{corollary}[theorem]{Corollary}

\newtheorem{remark}[theorem]{Remark}

\newtheorem{definition}[theorem]{Definition}

\newcommand{\R}{\mathbb{R}}

\newcommand{\grad}{\mbox{grad}}

\newcommand{\iis}{isometric immersions }

\newcommand{\h}{\mathcal{H}}

\def\P{{\cal P}}

\def\<{\langle}

\def\>{\rangle}

\def\bea{\begin{eqnarray*} }
\def\eea{\end{eqnarray*} }
\def\be{\begin{equation}}
\def\ee{\end{equation}}

\def\qed{\ifhmode\unskip\nobreak\fi\ifmmode\ifinner
\else\hskip5 pt \fi\fi\hbox{\hskip5 pt \vrule width4 pt
height6 pt  depth1.5 pt \hskip 1pt }}
\begin{document}
\title{Minimal $(n-2)$-umbilic submanifolds of the Euclidean space}
\author{A.E.\ Kanellopoulou}
\date{}
\maketitle
\renewcommand{\thefootnote}{\fnsymbol{footnote}} 
\footnotetext{\emph{The paper is part of the author's PhD thesis, written in the Department of Mathematics at the University of Ioannina under the supervision of Th. Vlachos.}}     
\renewcommand{\thefootnote}{\arabic{footnote}} 

\renewcommand{\thefootnote}{\fnsymbol{footnote}} 
\renewcommand{\thefootnote}{\arabic{footnote}}

\renewcommand{\thefootnote}{\fnsymbol{footnote}} 

\begin{abstract}

This paper investigates minimal $n$-dimensional submanifolds in the Euclidean space that are $(n-2)$-umbilic, meaning they carry an umbilical distribution of rank $n-2$. We establish a correspondence between the class of minimal ($n-2$)-umbilic submanifolds and the class of ($n-2$)-singular minimal surfaces. These surfaces are the critical points of its "energy potential" and have been previously studied in various contexts, including physics and architecture where, for instance, they model surfaces with minimal potential energy under gravitational forces. We show that minimal, generic, ($n-2$)-umbilic submanifolds, $n\geq4$, are ($n-2$)-rotational submanifolds whose profile is an ($n-2$)-singular minimal surface and vise versa. Furthermore, we develop a Weierstrass type method of local parametrization of all ($n-2$)-singular minimal surfaces, enabling a parametric description of all  minimal $n$-dimensional, $n\geq4$, hypersurfaces of the Euclidean space with a nowhere vanishing principal curvature of multiplicity $n-2$.

\end{abstract}

 \section{Introduction}

It is a classical result that the only surface of revolution in the three dimensional Euclidean space that is minimal is the catenoid. As a generalization of the catenoid, D. E. Blair \cite{DEB} characterizes minimal conformally flat hypersurfaces of ($n+1$)-dimensional Euclidean space. It is well known that conformally flat hypersurfaces have a principal curvature of multiplicity $n-1$ at every point (see \cite{DODA}). Blair proved that the only minimal hypersurfaces of this type are the generalized catenoid  which is a rotational submanifold whose profile curve arises as solution of a specific ordinary differential equation.
 
We focus on generalizing Blair's result by studying minimal hypersurfaces in the Euclidean space $\mathbb{R}^{n+1}$, which possess a principal curvature of multiplicity $n-2$. Moreover, we are able to produce results for a broader class of minimal submanifolds in arbitrary codimension, specifically those belonging  to the class of $k$-umbilic submanifolds. We recall that an isometric immersion is called $k$-umbilic if it carries a maximal $k$-dimensional totally umbilic distribution. This means that there exists a smooth vector field $\eta\in N_fM$ of unit length and a positive function $\mu\in\mathbb{C}^\infty(M)$ so that:
$$
\mathcal{U}=\left\{X\in T_fM:\alpha^f(X,Y)=\mu \<X,Y\>\eta \;\;\text{for every}\;\;Y\in TM\right\}.
$$
 More presicely, we study the class of $n$-dimensional minimal ($n-2$)-umbilic submanifolds. 
 
An important class of {\it ($n-2$)-umbilic submanifolds} is the class of {\it ($n-2$)-rotational submanifolds}. To our surprise, minimal ($n-2$)-umbilic submanifolds turn out to be ($n-2$)-rotational submanifolds over a special class of surfaces, the {\it ($n-2$)-singular minimal surfaces}. In general, an {\it $a$-singular minimal surface} with {\it density vector $e$} and $a\in\mathbb{R}$, is by definition an isometric immersion $g\colon L^2\to\mathbb{R}^{p+2}$ of a two dimensional Riemannian manifold $L^2$ whose mean curvature vector field $\h_g$ satisfies the condition $2\h_g=a e^{\perp}/\phi$, where $e^\perp$ is the orthogonal projection of the density vector $e$ onto the normal bundle of $g$ and $\phi=\<g, e\>>0$. 

The concept of singular minimal surfaces emerged in the 1980s as a significant topic in geometric analysis, originating from the study of the isoperimetric problem in $\mathbb{R}^3$. Specifically, it arose in the search for a surface $L^2$ that minimizes potential energy under gravitational forces, that is, a surface with the lowest center of gravity (see \cite{domes}). After the derivation of the law of refraction of light by Pierre Fermat in 1662, the variational problem of minimizing the functional
$$
\int \omega(x, y, z)\sqrt{dx^2+dy^2+dz^2},
$$
gained prominence. According to Fermat's principle, this
variational problem governs the propagation of light in an isotropic
but inhomogeneous medium where $1/\omega$ is the velocity of the light
particle. In 1696, John Bernoulli proposed the problem of quickest descend, also known as {\it the brachystochrone}, which consists in minimizing the integral 
$
\int 1/z\sqrt{dx^2+dz^2}
$.
In 1744, Leonhard Euler worked on a more general case, where $1/z= \omega(x, z)$. The integral $\int z\sqrt{dx^2+dz^2}$, where $z>0$, leads to the celebrated problem, to determine the surfaces of revolution minimizing area. The regular extremals of this integral are the catenaries given by
$$
z=a\cosh\frac{x-x_0}{a}.
$$
It turns out that the catenary describes also the equilibrium position of a heavy chain. In fact, it minimizes potential energy under
the influence of gravity force, in other words it has the lowest center of gravity.
The isoperimetric problem to determine  a surface having the
lowest center of gravity was dealt with by R. Buhme, S. Hildebrandt and E. Tausch in \cite{domes}. They also show in this paper that the importance of such surfaces is also in the construction of {\it "perfect domes"}. According to U. Dierkes in \cite{DRK2}, the mean curvature equation describes the shape of a hanging roof and in a more general context, membranes under their own weight, known as "heavy surfaces of minimal surface type". This was studied by Lagrange and Poisson and is still of interest in modern architectural design, particularly in the context of designing tension structures. 
The analogue for the one-dimensional problem is due to Gauss.

The notion of $a$-singular minimal surfaces in $\mathbb{R}^3$ is a generalization of the two-dimensional analogue of the catenary, in which case $a=1$. 

The following theorem establishes the relationship between the class of minimal $n$-dimensional $(n-2)$-umbilic submanifolds, the class of ($n-2$)-rotational submanifolds and the class of singular minimal surfaces. Here, we need to remind that a $k$-umbilic submanifold $f$ is called {\it generic} when it satisfies the following
$$
\dim[\ker(A_{\eta}-\mu I)(x)]=k, \,\, \text{for all} \, \, x\in M^{n},
$$
where $A_{\eta}$ is the shape operator associated to the unit smooth vector field $\eta\in N_fM$.

\begin{theorem}\label{sms0}
The following assertions hold:
\begin{enumerate} 
\item[(i)] Any minimal $(n-2)$-umbilic isometric immersion $f\colon M^n\to \mathbb{R}^{n+p},\, n\geq 4$, which is generic is an $(n-2)$-rotational submanifold.
\item[(ii)] An $(n-2)$-rotational submanifold $f\colon M^n\to \mathbb{R}^{n+p}$, $n\geq3$, is minimal if and only if its profile $g\colon L^2\to\mathbb{R}^{p+2}$ is an $a$-singular minimal surface with $a=n-2$. 
\end{enumerate}
\end{theorem}

\medskip

The relationship established in the above theorem has been proven to be essential in our research since it allows us to obtain a parametric description of the class of minimal $n$-dimensional ($n-2$)-umbilic submanifolds through the study of $(n-2)$-singular minimal surfaces. 
To this aim, it is crucial to prove that singular minimal surfaces are indeed minimal surfaces in the Euclidean space endowed with a conformal singular metric. This justifies the terminology of singular minimal surfaces. More precisely, we prove the following theorem.

\begin{theorem}\label{Ilamnen}
An isometric immersion $g$ of a two-dimensional Riemanninan manifold $L^2$ into 
$(\mathbb{R}_+^{p+2},\<\cdot,\cdot\>)$, where $\mathbb{R}_+^{p+2}=\{x\in\mathbb{R}^{p+2}/\sigma(x)>0\}$, $\sigma(x)=\<x,e\>$, and $e$ is a unit vector,
is $a$-singular minimal if and only if $g$ is a minimal surface, viewed as an isometric immersion into the Riemannian manifold $(\mathbb{R}^{p+2}_{+},\<\cdot,\cdot\>_{a})$, endowed with the metric $\<\cdot,\cdot\>_{a}=\sigma^a\<\cdot,\cdot\>$, being $\<\cdot,\cdot\>$ the flat Euclidean metric.
\end{theorem}

The metric $\<\cdot,\cdot\>_{a}$ given above first appeared in a work by  U. Dierkes \cite{DRK2}.

Focusing on this special class of surfaces in the Euclidean space, we provide a method, similar to the Weierstrass parametrization, that parametrizes locally all substantial $(n-2)$-singular minimal surfaces. Being substantial means that the codimension cannot be reduced.

\begin{theorem}\label{WT0}
Let  $g\colon (L^2,\<\cdot,\cdot\>)\to \mathbb{R}^{p+2}_+,p\geq1,$ be a substantial ($n-2$)-singular minimal surface with density vector $e$, where $L^2$ is an oriented two-dimensional Riemannian manifold. Then, for any simply connected complex chart $(U, z)$ of $(L^2,\<\cdot,\cdot\>)$, there exists a non-holomorphic map $G=\sum_{i=1}^p G_ie_i\colon U\to\mathbb{C}^p$, with $G_p\neq0$, where $\{e_1, \dots, e_{p-1}, e_p=e\}$ is an orthonormal basis of $\mathbb{C}^p$, that satisfies: 
\begin{equation}\label{eq0}
G_{\bar{z}z}=-\frac{n}{n-2}\frac{||G_{\bar{z}}||^2}{||F(G)||^2}G_pF(G)+\frac{||G_{\bar{z}}||^2}{||F(G)||^2}\<G,G\>\bar{G}_pF(G)+\frac{\<G_{\bar{z}},\overline{F(G)}\>}{||F(G)||^2}(F(G))_z
\end{equation}
 and 
 \begin{equation}\label{cond0}
 G_{\bar{z}}\wedge F(G)=0,  
\end{equation}
 where $F(G)\neq0$ and $F\colon\mathbb{C}^p\to\mathbb{C}^p$ is given by  
$$
F(z)=\left(\frac{1+|\<z,z\>|^2}{2}+||z||^2\right)e_p-z_p\big(\<\bar{z},\bar{z}\>z+\bar{z}\big),\, \quad z=(z_1,\cdots, z_p)\in\mathbb{C}^p.
$$
Moreover, the immersion is given by

\begin{multline}\label{par0}
g={\rm Re}\Big(\int\frac{\<\bar{G}_z,F(G)\>}{||F(G)||^2}(1-\<G,G\>)hdz, \mathsf{i}\int\frac{\<\bar{G}_z,F(G)\>}{||F(G)||^2} (1+\<G,G\>)hdz,\\ 2\int \frac{\<\bar{G}_z,F(G)\>}{||F(G)||^2} G_1hdz,\dots, 2\int \frac{\<\bar{G}_z,F(G)\>}{||F(G)||^2} G_{p-1}hdz,(n-2)h\Big),
\end{multline}
where 
$$
h=\frac{1}{n-2}\left(\frac{n}{2}\right)^\frac{2}{n}e^{\frac{2}{n-2}{\rm Re}\int\frac{\<\bar{G}_z,F(G)\>}{||F(G)^2||}G_pdz}.
$$

Conversely, consider a non-holomorphic map $G=(G_1,\dots, G_p)\colon U\to\mathbb{C}^p$, such that $G_p\neq0$ and $F(G)\neq0$ on a simply connected open subset $U\subset\mathbb{C}$, that satisfies (\ref{eq0}) and (\ref{cond0}). Then, the surface $g\colon L^2\to \mathbb{R}^{p+2},p\geq1,$ parametrized by (\ref{par0}) is an ($n-2$)-singular minimal surface with density vector $e=(0,\dots, 0,1).$
\end{theorem}

It is worth noticing that in the lowest allowed codimension $p=1$, the condition \eqref{cond0} in the above theorem is trivially satisfied.  
In fact, for this case, the Weierstass type representation given in the above theorem was already given in \cite{MM}.   
Minimal hypersurfaces in the Euclidean space $\mathbb{R}^{n+1}$ that have a nowhere vanishing principal curvature of multiplicity $n-2$ are clearly generic ($n-2$)-umbilic submanifolds. Combining Theorem \ref{sms0} with Theorem \ref{WT0} for $p=1$, we obtain the following local result. It is intresting to point out that in contrast to the case studied by D. E. Blair, where there is a unique minimal ($n-1$)-umbilic hypersurface, our result demonstrates that there is a rich class of minimal ($n-2$)-umbilic hypersurfaces that are geometrically distinct.

\begin{corollary}
Any minimal hypersurface in the Euclidean space $\mathbb{R}^{n+1}$, $n\geq4$, that has a nowhere vanishing principal curvature of multiplicity $n-2$, is an ($n-2$)-rotational hypersurface whose profile is an ($n-2$)-singular minimal surface in $\mathbb{R}^3$ which is parametrically given by the above theorem.
\end{corollary}

Additionally, we obtain the following local parametrization of flat ruled $(n-2)$-singular minimal surfaces.

\begin{theorem}\label{frsms}
The only substantial $(n-2)$-singular minimal flat ruled surfaces $g\colon L^2\to \mathbb{R}^{p+2}$ are the cylinders over a plane curve $c$, which is not a circle, and its curvature $k$ satisfies the differential equation
\begin{equation*}
\dot{k}^2 + \left(\frac{n-1}{n-2}\right)^2k^4-\left(\frac{n-1}{\lambda}\right)^2k^{2n/(n-1)}=0,
\end{equation*}
where $\lambda$ is a nonzero constant. Conversely, the cylinder over a plane curve whose curvature is a nonconstant solution of the above ODE is an ($n-2$)-singular minimal surface.
 \end{theorem}

\medskip

Finally, we provide a global result on $(n-2)$-singular minimal surfaces. Namely, we prove that there are no complete $(n-2)$-singular minimal surfaces with Gaussian curvature bounded from bellow that are contained between two hyperplanes perpendicular to the density vector.

The paper is organized as follows: In Section 2, we recall some basic notions of the theory of submanifolds and give a brief overview of the background material needed for the rest of the paper.  Section 3 focuses on our study of minimal $n$-dimensional ($n-2$)-umbilic submanifolds. Since minimal generic $(n-2)$-umbilic submanifolds are indeed $(n-2)$-rotational submanifolds over ($n-2$)-singular minimal surfaces (Theorem \ref{sms0}), we need to turn our attention to $a$-singular minimal surfaces. We give the definition of $a$-singular minimal surfaces with density vector and we prove that they constitute the critical points of their energy potential for any variation. We prove one of our main results, namely Theorem \ref{Ilamnen}, that claims that a surface is an $a$-singular minimal surface if and only if it is a minimal surface in the Riemannian manifold $(\mathbb{R}_+^{p+2},\<\cdot,\cdot\>_a$). In Section 4, we prove Theorem \ref{sms0} that allows us to proceed our research on a local classification of minimal $(n-2)$-umbilic submanifolds through the local description of ($n-2$)-singular minimal surfaces. Additionally, we prove Theorem \ref{frsms} that provides a local parametrization of flat ruled $(n-2)$-singular minimal surfaces and we prove a result of global character, that is, the non-existence of complete ($n-2$)-singular minimal surfaces. In Section 5, we develop a method similar to Weierstass representation, in order to prove our main theorem, Theorem \ref{WT0} and thus achieve the local parametrization of all ($n-2$)-singular minimal surfaces.

\section{Preliminaries}

In this section, we recall some basic facts from the theory of \iis
that will be used throughout the paper.
 
\subsection{$k$-umbilic submanifolds}
  
An isometric immersion $f\colon M^n\to\R^{n+p}$ is called {\it $k$-umbilic} if it carries a maximal $k$-dimensional totally umbilic distribution $\mathcal{U}$. This means that there exists a smooth vector field $\eta\in N_fM$ of unit length and a positive function $\mu\in\mathcal{C}^\infty(M)$ so that:
$$
\mathcal{U}=\left\{X\in TM:\alpha^f(X,Y)=\mu \<X,Y\>\eta \;\;\text{for every}\;\;Y\in TM\right\}.
$$
It is known that the distribution $\mathcal U$ is integrable.
Additionally, a $k$-umbilic submanifold $f$ is called {\it generic} when it satisfies the following
$$
\dim[\ker(A_{\eta}-\mu I)(x)]=k, \,\, \text{for all} \, \, x\in M^{n},
$$
where $A_{\eta}$ is the shape operator associated to $\eta$ given by
$$
\<A_{\eta}X,Y\>=\<\alpha^f(X,Y),\eta\>, \,\  \text{for all}\,\,   X,Y\in TM.
$$
Observe that every $k$-umbilic hypersurface is trivially generic. 

\subsection{($n-2$)-rotational submanifolds}

An important class of $k$-umbilic sumbmanifolds with $k=n-2$ are the {\it $(n-2)$-rotational} submanifolds.
An $(n-2)$-rotational submanifold $f\colon M^n\to\mathbb{R}^{n+p}, n\geq 3,$ with rotation axis $\mathbb{R}^{p+1}$ over a surface  $g\colon L^2\to \mathbb{R}^{p+2}$ is the $n$-dimensional submanifold generated by the orbits of the points of $g(L)$ (disjoint from $\mathbb{R}^{p+1}$) under the action of the subgroup $SO(n-1)$ of $SO(n+p)$ which keeps pointwise $\mathbb{R}^{p+1}$ invariant. Any $(n-2)$-rotational submanifold $f\colon M^n\to\mathbb{R}^{n+p}, n\geq 3$, can be parametrized as follows. The manifold $M^n$ is isometric to (an open subset of) a warped product $ L^2\times_{\phi}\mathbb{S}^{n-2}(1)$ and there is an orthogonal splitting  $\mathbb{R}^{p+2}=\mathbb{R}^{p+1}\oplus span\{e\}$, $ ||e||=1$, such that the {\it profile} $g$ has the form $ g=(h,\phi e)$,  where $h\colon L^2\to\mathbb{R}^{p+1}$ and $\phi=\<g,e\>>0$. Then, 
$$
f(x,y)=\left(h(x),\phi(x) j(y)\right),
$$
where $j\colon\mathbb{S}^{n-2}\to\mathbb{R}^{n-1}$ is the inclusion map of the unit sphere.

\section{Singular minimal surfaces}

For a fixed vector $e\in\mathbb{R}^{p+2}$, we consider the half space $\mathbb{R}_+^{p+2}=\{x\in\mathbb{R}^{p+2}\colon\,\sigma(x)>0\}$ where  $\sigma\colon\mathbb{R}^{p+2}\to\mathbb{R}$ is the height function given by $\sigma(x)=\<x,e\>$, where $\<\cdot,\cdot\>$ denotes the standard inner product in $\mathbb{R}^{p+2}$.
\begin{definition}
An isometric immersion $g\colon (L^2,\<\cdot,\cdot\>)\to(\mathbb{R}^{p+2},\<\cdot,\cdot\>)$ is called {\it a-singular minimal surface} with density vector $e$, where $e$ is a unit vector and $a\in\mathbb{R},$ if and only if $g(L)\subset\mathbb{R}^{p+2}_{+}$ and its mean curvature vector field satisfies
$$
2\h_g=\frac{a}{\phi}e^{\perp},
$$
where $\phi=\sigma\circ g$ and $e^\perp$ is the orthogonal projection of $e$ onto the normal bundle $N_gL$ of $g$.

\end{definition}

Let $g\colon L^2\to \mathbb{R}^{p+2}$ be an isometric immersion such that $\phi=\<g,e\>>0,$ where $e$ is a fixed vector and let $G\colon (-\varepsilon,\varepsilon)\times L^2\to\mathbb{R}^{p+2}$ be a smooth variation of $g$. This means that the map $g_t\colon L^2\to\mathbb{R}^{p+2}$ given by $g_t(x)=G(t,x)$ is an immersion with $g_0=g$. For a given $a\in\mathbb{R}$ we define the {\it energy potential} of $g$ associated to the smooth variation and the vector $e$ as follows
$$
E _{a,e}(g)=\int_{L^2}\phi^adL_t,
$$
where $dL_t$ is the volume element of the metric induced by $g_t$. 

The following proposition provides a variational characterization of $a$-singular minimal surfaces.
\begin{proposition}
 Let $g\colon L^2\to \mathbb{R}^{p+2}$ be an isometric immersion with $\phi=\<g,e\>>0$ where $e$ is a fixed vector. For any compactly supported smooth variation $g_t$ of $g$, the first variation of the energy potential is given by:
 $$
\frac{d}{dt}|{_{t=0}}\left(E_{a,e}(g_t)\right)=\int_{L^2}\phi^{a-1}\<ae^{\perp}-2\phi \h_g,\eta\>dL,
$$
where $\eta\in\Gamma(N_gL)$ is the normal component of the variational vector field $\partial g_t/\partial t|_{t=0}$. 

Moreover, $g$ is a critical point of the energy potential for any compactly supported variation if and only if $g$ is an $a$-singular minimal surface with density vector $e$.
\end{proposition}

\begin{proof}
Let $\partial g_t/\partial t|_{t=0}=dg(Z)+\eta$ be the composition of the variational vector field into its tangent and normal components.
By using Proposition 3.1 in \cite{da}  we have that
\begin{align*}
\frac{d}{dt}|{_{t=0}}\left(E_{a,e}(g)\right)&=\int_{L^2}\left(\frac{d}{dt}|{_{t=0}}\<g_t,e\>^adL+\<g,e\>^a\frac{d}{dt}|{_{t=0}}(dL_t)\right)\\
&=\int_{L^2}\left(a\phi^{a-1}\<dg(Z)+\eta,e\>dL+\phi^a(-2\<\h_g,\eta\>+\operatorname{div}ZdL)\right)\\
&=\int_{L^2}\phi^{a-1}\<ae^\perp-2\phi \h_g,\eta\>dL,
\end{align*}
and this completes the proof.
\qed
\end{proof}

\medskip

We now consider a half space $$\mathbb{R}_+^n=\{x\in\mathbb{R}^n\colon \sigma(x)=\<x,e\>>0\},$$ where $e$ is a unit vector, as an $n$-dimensional Riemannian manifold endowed with the conformal metric $$\<\cdot,\cdot\>_{a}=\sigma^{a}\<\cdot,\cdot\>, \, \, a\in\mathbb{R}.$$

\begin{theorem} \label{minimal}
An isometric immersion $g\colon (L^2,g^*\<\cdot,\cdot\>)\to(\mathbb{R}_+^{p+2},\<\cdot,\cdot\>)$
is $a$-singular minimal if and only if $g$ is a minimal, viewed as an isometric immersion into the Riemannian manifold  $(\mathbb{R}^{p+2}_{+},\<\cdot,\cdot\>_{a})$.
\end{theorem}

\begin{proof}
Since the metric $\<\cdot,\cdot\>_{a}$ is conformal to the standard metric of $\mathbb{R}^{p+2}_{+}$, it is well known that (see for instance \cite[p. 304]{da}) the second fundamental forms $\alpha_1$ and $\alpha_2$ of the isometric immersions $g\colon (L^2,g^*(\<\cdot,\cdot\>))\to(\mathbb{R}_+^{p+2},\<\cdot,\cdot\>)$ and $g\colon (L^2,g^*(\<\cdot,\cdot\>_{a}))\to(\mathbb{R}^{p+2}_{+},\<\cdot,\cdot\>_{a})$, respectively, are related by
$$
\alpha_2(X,Y)=\alpha_1(X,Y)-\frac{a}{2\phi}\<X,Y\>e^\perp, \, X,Y\in TL^2,  
$$
where, by abuse of notation, $\<\cdot,\cdot\>$ is the induced metric $g^*(\<\cdot,\cdot\>)$.
Hence, the corresponding mean curvature vector fields $\h_1$ and $\h_2$ of the isometric immersions satisfy
$$
2\h_2=\frac{1}{\phi^a}\left(2\h_1-\frac{a}{\phi}e^\perp\right).
$$
This completes the proof.\qed

\end{proof}

\section{Minimal $(n-2)$-umbilic submanifolds}

We provide now the proof of one of the basic theorems of the paper which shows that the minimal, generic ($n-2$)-umbilic submanifolds are actually ($n-2$)-rotational over ($n-2$)-singular minimal surfaces.

\vspace{2ex}

\begin{theorem}\label{sms}
The following assertions hold:
\begin{enumerate} 
\item[(i)] Any minimal generic $(n-2)$-umbilic, isometric immersion $f\colon M^n\to \mathbb{R}^{n+p},\, n\geq 4$, is an $(n-2)$-rotational submanifold.
\item[(ii)] An $(n-2)$-rotational submanifold $f\colon M^n\to \mathbb{R}^{n+p}$, $n\geq3$, is minimal if and only if its profile $g\colon L^2\to\mathbb{R}^{p+2}$ is an $a$-singular minimal surface, with $a=n-2$. 
\end{enumerate}
\end{theorem}
\begin{proof}
$(i)$ The first assertion is an immediate consequence of Theorem 2 in \cite{DF1}, due to the minimality of $f$.

$(ii)$ Let $f\colon M^n= L^2\times_{\phi}\mathbb{S}^{n-2}\to\mathbb{R}^{n+p}, n\geq 3,$ be an  $(n-2)$-rotational submanifold with profile $g\colon L^2\to \mathbb{R}^{p+2}$. Observe that if $\xi\in N_gL(x)$, then $\xi-\<\xi,e\>e +\<\xi,e\>y \in N_fM(x,y)$ for any $(x, y) \in L^2\times\mathbb{S}^{n-2}$. Then, it is easy to see that the map $\P\colon N_gL\to N_fM$ given by
$$
\P_{(x,y)}(\xi)=(\xi-\<\xi,e\>e) +\<\xi,e\>y
$$ 
is a vector bundle isometry. Using the fact that the induced metric on $M^n$ is the warped product metric $\<\cdot,\cdot\>_{M^n}=\<\cdot,\cdot\>_{L^2}+\phi^2\<\cdot,\cdot\>_{\mathbb{S}^{n-2}}$ and the Weingarten formulas, we find that the shape operators $A_{\P\xi}^f$ and $A^g_\xi$ of $f$ and $g$ associated to $\P(\xi)$ and $\xi$, respectively, satisfy 
\begin{align*}
&\<A_{\P\xi}^fX_1,X_2\>_{M^n}=\<A^g_\xi X_1,X_2\>_{L^2},\quad X_1,X_2\in TL,\\
&\<A_{\P\xi}^fX,Y\>_{M^n}=0,\quad X\in TL,\, Y\in T\mathbb{S}^{n-1},\\
&\<A_{\P\xi}^fY_1,Y_2\>_{M^n}=-\frac{\<\xi,e\>}{\phi}\<Y_1,Y_2\>_{M^n},\quad Y_1,Y_2\in T\mathbb{S}^{n-1}.
\end{align*}
Let $\{\xi_1,\dots,\xi_p\}$ be a local orthonormal frame of  $N_gL$. Then $\{\P(\xi_1),\dots,\P(\xi_p)\}$ is an orthonormal frame of $N_fM$. Thus, the mean curvature of $f$ is given by
\begin{align*}
n\h_f&=\sum^p_{a=1}{\rm tr}A^f_{\P\xi_a}\P(\xi_a)=\P\left(\sum^p_{a=1}\left({\rm tr}A^g_{\xi_a}-\frac{(n-2)\<\xi,e\>}{\phi}\right)\xi_a\right)\\
&=\P \left(2\h_g-\frac{n-2}{\phi}e^\perp\right),
\end{align*}
where $\h_g$ is the mean curvature of $g$. This immediately implies that $f$ is minimal if and only if $g$ is an $a$-singular minimal surface with $a=n-2$.\qed
\end{proof}

\medskip

\vspace{2ex}
Now we focus on the class of ($n-2$)-singular minimal surfaces in order to provide some results of local and global character. 

\begin{theorem}\label{cylpar}
The only substantial $(n-2)$-singular minimal flat ruled surfaces $g\colon L^2\to \mathbb{R}^{p+2}$ are the cylinders over a plane curve $c$, which is not a circle, and its curvature $k$ satisfies the differential equation
\begin{equation}\label{pde}
\dot{k}^2 + \left(\frac{n-1}{n-2}\right)^2k^4-\left(\frac{n-1}{\lambda}\right)^2k^{2n/(n-1)}=0,
\end{equation}
where $\lambda$ is a nonzero constant. Conversely, the cylinder over a plane curve whose curvature is a nonconstant solution of the ODE (\ref{pde}) is an ($n-2$)-singular minimal surface.
 \end{theorem}

\begin{proof}
Let  $g\colon L^2\to\mathbb{R}^{p+2}$ be a flat ruled  $(n-2)$-singular minimal surface with density vector $e$. Suppose that $g$ is parametrized by 
$$
g(s,t)=c(s)+tw(s), \quad (s,t)\in L=I\times \mathbb{R},
$$
where $c\colon I\to\mathbb{R}^{p+1}$ is a unit speed curve with Frenet frame  $\{\dot{c}, N_1,\dots,N_p\}$ and $w$ is a unit vector field along $c$ perpendicular to $\dot{c}$.
Since $g$ is flat we know that $\dot{w}$ is parallel to $\dot{c}$, namely $\dot{w}(s)=a(s) \dot{c}(s)$, for a smooth function $a(s), s\in I$. We distinguish the following cases:
\medskip

\noindent\emph{Case I.} Suppose that $a=0$. Hence, $w$ is constant and $g$ is a cylinder. The normal bundle of $g$ is given by $N_gL={\rm span}\{N_1,\dots,N_p\}$.  Then we find that the shape operators $A_{N_i}$ of $g$ satisfy
$$
A_{N_1}\frac{\partial}{\partial s}=k_1\frac{\partial}{\partial s}, \quad A_{N_1}\frac{\partial}{\partial t}=0 \quad \text{and} \quad  A_{N_i}=0, i=2,\dots,p,
$$
where $k_1>0$ is the first Frenet curvature of $c$. Thus, $2\h_g=k_1N_1$. Then, the condition of being $(n-2)$-singular minimal is equivalent to the following equations
$$
\<e,N_i\>=0, i=2,\dots,p\ \, \ \text{and} \ \, \ k_1\left(\<c,e\>+t\<w,e\>\right)=(n-2)\<e,N_1\>.
$$
Hence, $\<w,e\>=0$ and thus, the functions  $u=\<c,e\>$ and $v=\<e,N_1\>$ satisfy
\begin{align}\label{uv}
e=\dot{u}\dot{c}+vN_1\\
k_1u=(n-2)v.\label{uvvv}
\end{align}
By differentiating equation \eqref{uv} and using the Frenet formulas, we obtain
\begin{equation}\label{uvk}
\ddot{u}=vk_1
\end{equation}
\begin{equation}\label{k2}
k_1\dot{u}=-\dot{v}, \  vk_2=0.
\end{equation}
Using equations \eqref{uv}, \eqref{uvvv} and the last equation of \eqref{k2}, we conclude that $k_2=0$. Thus $c$ is a plane curve and since $g$ is substantial we have that $p=1$. 
Differentiating equation \eqref{uvvv} and using the first equation of \eqref{k2}, we find
\begin{equation}\label{ert}
(n-1)k_1\dot{u}+\dot{k_1}u=0.
\end{equation}
Then, $(k_1u^{n-1})\dot{}=0$ and thus, 
\begin{equation}\label{k1}
u=\lambda k_1^{-\frac{1}{n-1}},
\end{equation}
where $\lambda\neq0$ is a constant. It follows from equation \eqref{uvvv} that
\begin{equation}\label{ert2}
 v=\frac{\lambda}{n-2}k_1^\frac{n-2}{n-1}.
\end{equation}
We claim that $c$ is not a portion of a circle. If otherwise, then the above two equations imply that both $u$ and $v$ are constant functions, which contradicts equation \eqref{uvk}. 
Equation \eqref{uv} implies that 
\begin{equation*}
\dot{u}^2+v^2=1.
\end{equation*}
Using equations \eqref{k1} and \eqref{ert2}, then it follows from the above that $k_1$ satisfies equation \eqref{pde}. 

Conversely, let $k_1$ be the curvature of a plane curve $c$ that is a non-constant solution of \eqref{pde}. We consider the functions $u$ and $v$ given by \eqref{k1} and \eqref{ert2}, respectively.
Using that $k_1$ solves the differential equation \eqref{pde}, we find that $\dot{u}\dot{c}+vN_1$ is a constant unit vector $e$, where $N_1$ is the unit normal vector of $c$. Moreover, we have that $k_1u=(n-2)v$. We now consider the cylinder in $\mathbb{R}^3$ over the plane curve $c$.  Then, it is easy to see that the cylinder is an $(n-2)$-singular minimal surface.

\medskip

\noindent\emph{Case II.} Suppose that $a$ is a nonzero constant. Then, the surface is a cone, and without loss of generality, we may suppose that it is parametrized by $g(s,t)=tc(s)$, where $c$ is a unit speed curve in $\mathbb{S}^{p+1}$ with Frenet frame $\{\dot{c}, N_1,\dots,N_p\}$ and Frenet curvatures $\{k_1,\dots,k_{p}\}$. It is easy to see that the induced metric of $g$ is the warped product $\<\cdot,\cdot\>_g=t^2ds^2+dt^2$. As in Case I, we find that $2\h_g=k_1N_1.$
Since $g$ is an $(n-2)$-singular minimal surface with density vector $e$ then
$$
k_1=\frac{n-2}{t\<c,e\>}\<e,N_1\>,
$$
which is obviously a contradiction.

\medskip

\noindent\emph{Case III.} Suppose that $a\dot{a}\neq0$. Then, we may reparametrize the surface $g$ such that $g(s,t)=c(s)+t\dot{c}(s)$, where $c$ is a unit speed curve with Frenet frame  $\{\dot{c}, N_1,\dots,N_p\}$ and Frenet curvatures $\{k_1,\dots, k_{p+1}\}$. Then the induced metric of $g$ is given by 
$$
\<\cdot,\cdot\>_g=(1+t^2k_1^2)ds^2+2dsdt+dt^2.
$$
Using that $\Delta g=2\h_g$,
where $\Delta g$ is the Laplacian of $g$, we find that

$$
2\h_g=\frac{1}{tk_1}\left((t^2+1) \dot{c} + \left(\frac{t\dot{k}_1+tk^3_1-t^2k_1-k_1}{tk_1}\right)N_1+k_2N_2\right).
$$
Since $g$ is an $(n-2)$-singular minimal surface with density vector $e$, we have $2\h_g=\frac{n-2}{\phi}e^\perp$. Taking the $\dot{c}$ component, we clearly reach a contradiction and this completes the proof.\qed
\end{proof}

\begin{remark}{\em The existence of non-flat ruled surfaces that are $(n-2)$-singular minimal is reduced to a rather involved system of ordinary differential equations.}
\end{remark}

\vspace{2ex}
\begin{proposition}\label{global}
There are no complete $(n-2)$-singular minimal surfaces with Gaussian curvature bounded from below, that are contained between two hyperplanes perpendicular to the density vector $e$.
\end{proposition}
\begin{proof}
Let $g\colon L^2\to \mathbb{R}^{p+2}$ be a complete $(n-2)$-singular minimal surface with density vector $e$ and Gaussian curvature bounded from below. At each point of $L^2$, we decompose the density vector $e$ into its tangent and normal components, namely $e=dg(e^{\top})+e^\perp$.  
Observe that the gradient of $\phi$ is given by $\grad\phi=e^\top$. Moreover, we have 
$$
\nabla_Xe^\top=A_{e^\perp}X,\  X\in TL,
$$
where $\nabla$ is the Levi-Civita connection of $L^2$ and $A_{e^\perp}$ is the shape operator of $g$ associated to the direction of $e^\perp$.
 Then, a direct computation shows that the Laplacian of the function $\phi=\<g,e\>$ is given by
\begin{equation*}
 \Delta\phi={\rm trace}(A_{e^\perp})=\<2\h_g,e^\perp\>=\frac{n-2}{\phi}\|e^\perp\|^2.
\end{equation*}
Using the above equation, we find
\begin{equation}\label{ineq}
\frac{1}{2}\Delta\phi^2=\phi\Delta\phi+\|\nabla\phi\|^2\geq1.
\end{equation}
Assume now, that $g$ is contained between two hyperplanes perpendicular to the density vector $e$. Then, $\phi$ is bounded and by the Omori-Yau maximum principle (see \cite{omori}) there exists a sequence $\{x_m\}$ of points in $L^2$ such that 
$$
{\rm lim}_{m\to\infty}\phi(x_m)={\rm sup}\, \phi \quad \text{and} \quad \Delta\phi(x_m)\leq1/m
$$
which clearly contradicts inequality \eqref{ineq} for $m$ large enough.\qed
\end{proof}

\begin{remark}{\em   It is worth pointing out the similarity between the condition of being $(n-2)$-singular minimal surface and the one being a translating soliton. However, $(n-2)$-singular minimal surfaces cannot be translating solitons. In fact,  it follows from the inequality \eqref{ineq} in the proof of Proposition \ref{global} that no open subset of an $(n-2)$-singular minimal surface is contained between two hyperplanes perpendicular to the density vector $e$. }
\end{remark}

\section{A representation theorem for $(n-2)$-singular minimal surfaces}
It turns out from Theorem \ref{sms} that a local classification of minimal $(n-2)$-umbilic submanifolds is reduced to the local description of all $(n-2)$-singular minimal surfaces.

 Our goal in this section is to describe, locally, all $(n-2)$-singular minimal surfaces in $\mathbb{R}^{p+2}$. Let $e\in\mathbb{R}^{p+2}$ be a unit  vector and $\mathbb{R}^{p+2}_+=\{x\in\mathbb{R}^{p+2}\colon \sigma(x)=\<x,e\>>0\}$. Let $(x_1, \dots, x_{p+2})\in\mathbb{R}^{p+2}$ be the coordinates with respect to  an orthonormal basis $\{e_1, \dots, e_{p+1}, e_{p+2}=e\}$. We now consider the Riemannian manifold $(\mathbb{R}^{p+1}\times\mathbb{R}_+,\widetilde{\<\cdot,\cdot\>})$ with coordinates $(x_1, \dots, x_{p+1}, \omega)\in\mathbb{R}^{p+1}\times\mathbb{R}_+$ equipped with the warped product metric given by 
$$
\widetilde{\<\cdot,\cdot\>}=e^{\eta} \sum_{i=1}^{p+1} d{x_i}^2+d\omega^2,
$$
where   
\begin{equation}\label{h}
\eta=\eta(\omega)=\frac{2(n-2)}{n}\log\left(\frac{n}{2}\omega\right).
\end{equation}
The map $\tau\colon\left(\mathbb{R}_+^{p+2}, \<\cdot,\cdot\>_{n-2}\right)\to\left(\mathbb{R}^{p+1}\times\mathbb{R}_+,\widetilde{\<\cdot,\cdot\>}\right)$ given by
$$
\tau(x_1, \dots, x_{p+2})=(x_1, \dots, x_{p+1}, \frac{2}{n}x_{p+2}^\frac{n}{2})
$$
is an isometry between the Riemannian manifold $(\mathbb{R}_+^{p+2}, \<\cdot,\cdot\>_{n-2})$ and the manifold $(\mathbb{R}^{p+1}\times\mathbb{R}_+,\widetilde{\<\cdot,\cdot\>})$.

Let $g\colon L^2\to(\mathbb{R}_+^{p+2}, \<\cdot,\cdot\>)$ be an isometric immersion of a two dimensional Riemannian manifold $L^2$. It turns out from Theorem \ref{minimal} and the above discussion that $g$ is an $(n-2)$-singular minimal surface with density vector $e$ if and only if the immersion $Y\colon L^2\to(\mathbb{R}^{p+1}\times\mathbb{R}_+,\widetilde{\<\cdot,\cdot\>})$ given by $Y=\tau\circ g$ is minimal. Here $g$ is seen as an immersion into the Riemannian manifold $(\mathbb{R}_+^{p+2}, \<\cdot,\cdot\>_{n-2})$.

\bigskip

\bigskip

\bigskip

\begin{picture}(150,84)\hspace{-15ex}
\put(140,30){$L^2$}
\put(120,55){$g$}\put(155,32){\vector(1,0){145}}
\put(305,30){$\left(\mathbb{R}^{p+1}\times\mathbb{R}_+,\widetilde{\<\cdot,\cdot\>}\right)$}
\put(205,45){$\circlearrowright$}
\put(180,20){$Y=\tau\circ g$}
\put(142,42){\vector(0,1){30}} 
\put(98,80){$\left(\mathbb{R}_+^{p+2}, \<\cdot,\cdot\>_{n-2}\right)$}
\put(180,80){\vector(3,-1){120}} \put(235,65){$\tau$}
\end{picture}

Let $Y\colon L^2\to(\mathbb{R}^{p+1}\times\mathbb{R}_+,\widetilde{\<\cdot,\cdot\>})$ be an isometric immersion.
For any simply connected chart $U$ of $L^2$ with complex coordinates $z=x+\mathsf{i}y$, such that the induced metric of the immersion $Y$ is given by $\<\cdot,\cdot\>_Y=\lambda^2(dx^2+dy^2)$, we consider the associated Wirtinger operators as follows:
$$
\partial=\frac{\partial}{\partial z}=\frac{1}{2}(\partial _{x}-\mathsf{i}\partial_y), \,\  \ \
\bar{\partial}=\frac{\partial}{\partial {\bar{z}}}=\frac{1}{2}(\partial_x+\mathsf{i}\partial_y).
$$
Now we define the complex functions $\varphi_i\colon U\subset L^2\to\mathbb{C}, 1\leq i\leq p+2$, by
\begin{equation*}
\varphi_i=e^\frac{\eta}{2}y_{i,z}, \   1\leq i\leq p+1 \quad
\text{and} \quad
\varphi_{p+2}=y_{p+2,z},
\end{equation*}
where $y_{i,z}=\frac{\partial y_i}{\partial z}$, for $1\leq p+2$, and $(y_1,\dots, y_{p+2})$ are the coordinate functions of $Y$. Clearly, we have
\begin{equation}\label{iii}
\sum_{i=1}^{p+2}\varphi_i^2=0,  \quad
\sum_{i=1}^{p+2}|\varphi_i|^2=\lambda^2/2.
\end{equation}

The functions $\varphi_i, 1\leq i\leq p+2$, determine locally the Gauss map of $g$ as a surface in the Euclidean space.

\begin{lemma}\label{fi}
An isometric immersion $Y\colon L^2\to(\mathbb{R}^{p+1}\times\mathbb{R}_+,\widetilde{\<\cdot,\cdot\>})$ is minimal if and only if for any complex chart $(U, z)$ of $L^2$ the associated complex functions  $\varphi_i, 1\leq i\leq p+2, $ satisfy the following conditions 

\begin{align}
& 2\varphi_{i,\bar{z}}=-\eta'  \bar{\varphi}_i\varphi_{p+2}, \,\, 1\leq i\leq p+1, \label{i}\\ 
&2\varphi_{p+2,\bar{z}}=\eta'\sum_{i=1}^{p+1}|\varphi_i|^2 \label{ii}. 
\end{align}
where $\eta'=\frac{d\eta}{d\omega}$.
\end{lemma}

\begin{proof}
Since the induced metric of the immersion $Y$ is given by $\<\cdot,\cdot\>_Y=\lambda^2(dx^2+dy^2)$, we have
\begin{equation}\label{iso}
\<\partial, \partial\>_Y=0, \,\ \ \  \<\partial, \bar{\partial}\>_Y=\lambda^2/2.
\end{equation}
Now, the minimality of $Y$ is equivalent to the following condition
\begin{equation}\label{minimality}
\tilde{\nabla}_\partial dY(\bar{\partial})=0,
\end{equation}
where $\tilde{\nabla}$ is the Levi-Civita connection of the manifold $(\mathbb{R}^{p+1}\times\mathbb{R}_+,\tilde{\<\cdot,\cdot\>})$.
On that manifold we consider the coordinate vector fields
$$
\partial_i=\frac{\partial}{\partial x_i},  \,\  1\leq i\leq p+1, \quad
\partial_{p+2}=\frac{\partial}{\partial\omega}.
$$
According to the above, equation \eqref{minimality} is equivalent to the following
\begin{equation}\label{minimality2}
0=\tilde{\nabla}_\partial\left( \sum^{p+2}_{i=1} y_{i,\bar{z}}\partial_i\right)=\sum^{p+2}_{i=1}y_{i,\bar{z}z}\partial_i + \sum^{p+2}_{i=1} y_{i,\bar{z}}\tilde{\nabla}_{dY(\partial)} \partial_i ,
\end{equation}
where
$$
\tilde{\nabla}_{dY(\partial)} \partial_i=\sum_{j=1}^{p+2}y_{i,z}\tilde{\nabla}_{\partial_j} \partial_i \quad
\text{and} \quad
\tilde{\nabla}_{\partial_j} \partial_i=\sum_{k=1}^{p+2}\Gamma_{ij}^k\partial_k.
$$
A direct computation yields that the Christoffel symbols corresponding to the coordinates $(x_1,\dots,x_{p+1},\omega)$ of the Riemannian manifold $(\mathbb{R}^{p+1}\times\mathbb{R}_+,\tilde{\<\cdot,\cdot\>})$, are given by

\[ \Gamma_{ij}^k=
\begin{cases}
\frac{1}{2}\eta',\ \, \quad \quad p+2=i \neq j=k\leq p+1 \\
-\frac{1}{2}e^\eta\eta', \quad i=j\leq p+1, k=p+2 \\
0, \quad \quad \quad \ \, \text{otherwise}.
 \end{cases}
\]
Then,  equation (\ref{minimality2}) is equivalent to the following system

\begin{align*}
&y_{i,\bar{z}z}+\frac{1}{2}\eta'\left(y_{i,\bar{z}}y_{p+2,z}+y_{i,z}y_{p+2,\bar{z}}\right)=0, \quad 1\leq i\leq p+1, \\
&y_{p+2,\bar{z}z}-\frac{1}{2}e^\eta\eta'\sum_{i=1}^{p+1}\left|y_{i,z}\right|^2=0.
\end{align*}

Taking into account the definition of the functions $\varphi_i, 1\leq i\leq p+2$, the first of the above equations is easily seen to be equivalent to equation \eqref{i}, whereas the second one is equivalent to \eqref{ii} in the statement. \qed
\end{proof}

\medskip

Let $Y\colon L^2\to(\mathbb{R}^{p+1}\times\mathbb{R}_+,\widetilde{\<\cdot,\cdot\>})$ be an isometric immersion. We assume that at least one of $\varphi_i$, $1\leq i\leq p+2$, is nonzero at each point. Without loss of generality, we may assume that the complex function
$$
\Psi=\varphi_1-\mathsf{i}\varphi_2,
$$
vanishes nowhere.
Then, the first equation of \eqref{iii}, is equivalently written as

$$
\varphi_1+\mathsf{i}\varphi_2 + \Psi\sum^p_{i=1}G_i^2=0,
$$ 
were $G_i=\varphi_i/\Psi$, for $1\leq i\leq p$. Hence, we have
$$\varphi_1=\frac{1}{2}\Psi\left(1-\sum^p_{i=1}G_i^2\right), \quad \varphi_2=\frac{\mathsf{i}
}{2}\Psi\left(1+\sum_{i=1}^pG_i^2\right)
$$
and thus
\begin{equation}\label{p}
(\varphi_1,\dots,\varphi_{p+2})=\frac{1}{2}\Psi\left(1-\<G,G\>, \mathsf{i}(1+\<G,G\>), 2G\right),
\end{equation}
where $G=(G_1,\dots,G_p)\colon U\to\mathbb{C}^p$.

\begin{lemma}\label{fi2}
An isometric immersion $Y\colon L^2\to(\mathbb{R}^{p+1}\times\mathbb{R}_+,\widetilde{\<\cdot,\cdot\>})$ is minimal if and only if the following conditions hold
\begin{align}
& \Psi_{\bar{z}}=\frac{1}{2}\eta'|\Psi|^2\<\bar{G},\bar{G}\>G_p,\label{Psizbar}\\ 
&G_{\bar{z}}=\frac{1}{2}\eta'\bar{\Psi}\left(\left(\frac{1+|\<G,G\>|^2}{2}+||G||^2\right)e_p-G_p\left(\<\bar{G},\bar{G}\>G+\bar{G}\right)\right), \label{Gz}\\
&\omega_z=\frac{\partial y_{p+2}}{\partial z}=\varphi_{p+2}=\Psi G_p.\label{wz}
\end{align}
\end{lemma}

\begin{proof}
Using \eqref{p}, equation \eqref{i} of Lemma \ref{fi} is written as
\begin{align}
(1-\<G,G\>)\Psi_{\bar{z}}-2\Psi\<G,G_{\bar{z}}\>&=-\frac{1}{2}\eta'|\Psi|^2(1-\<\bar{G},\bar{G}\>)G_p, \quad \text{for} \quad i=1,\nonumber \\
(1+\<G,G\>)\Psi_{\bar{z}}+2\Psi\<G,G_{\bar{z}}\>&=\frac{1}{2}\eta'|\Psi|^2(1+\<\bar{G},\bar{G}\>)G_p, \quad \text{for} \quad  i=2, \nonumber\\
2\Psi_{\bar{z}}G_{i-2}+2\Psi G_{i-2,\bar{z}}&=-\eta'|\Psi|^2\bar{G}_{i-2}G_p, \quad \text{for} \quad 3\leq i\leq p+1. \label{pro}
\end{align}
 The first two equations of the above are equivalent to
\begin{align}
&\Psi_{\bar{z}}=\frac{1}{2}\eta'|\Psi|^2\<\bar{G},\bar{G}\>G_p,\nonumber\\
\label{esoteriko}
&\<G,G_{\bar{z}}\>=\frac{1}{4}\eta'\bar{\Psi}(1-|\<\bar{G},\bar{G}\>|^2)G_p.
\end{align}
Similarly, equation \eqref{ii} becomes
\begin{equation} \label{tel}
\Psi_{\bar{z}}G_p+\Psi G_{p,\bar{z}}=\frac{1}{4}\eta'|\Psi|^2\left(1+|\<G,G\>|^2+2\|G\|^2-2|G_p|^2\right).
\end{equation}
Now, plugging \eqref{Psizbar} in \eqref{pro} and \eqref{tel} we obtain
\begin{align*}
&G_{j,\bar{z}}=\frac{1}{2}\eta'\bar{\Psi}G_p\left(\<\bar{G},\bar{G}\>G_j+\bar{G}_j\right), \quad 1\leq j\leq p-1,\\
&G_{p,\bar{z}}=\frac{1}{2}\eta'\bar{\Psi}\left(G_p\left(\<\bar{G},\bar{G}\>G_p+\bar{G}_p\right)-\|G\|^2-\frac{1+|\<G,G\>|^2}{2}\right),
\end{align*}
respectively. These two equations are equivalent to equation \eqref{Gz} and this completes the proof.\qed
\end{proof}

\medskip
It follows from Lemma \ref{fi} that if $g$ is a substantial ($n-2$)-singular minimal surface, then the complex functions $\varphi_i, 1\leq i\leq p+2$, associated to a complex chart, are non-holomorphic. This excludes a parametric description of $(n-2)$-singular minimal surfaces similar to the classical Weierstrass representation of minimal surfaces in terms of holomorphic data. However, it should be useful to provide a method that parametrizes locally all  $(n-2)$-singular minimal surfaces. This is achieved by the following theorem. 

\vspace{2ex}

\begin{theorem}\label{WT}
Let  $g\colon (L^2,\<\cdot,\cdot\>)\to \mathbb{R}^{p+2}_+,p\geq1,$ be a substantial ($n-2$)-singular minimal surface with density vector $e$, where $L^2$ is an oriented two-dimensional Riemannian manifold. Then, for any simply connected complex chart $(U, z)$ of $(L^2,\<\cdot,\cdot\>)$, there exists a non-holomorphic map $G=\sum_{i=1}^p G_ie_i\colon U\to\mathbb{C}^p$, with $G_p\neq0$, where $\{e_1, \dots, e_{p-1}, e_p=e\}$ is an orthonormal basis of $\mathbb{C}^p$, that satisfies: 
\begin{equation}\label{eq}
G_{\bar{z}z}=-\frac{n}{n-2}\frac{||G_{\bar{z}}||^2}{||F(G)||^2}G_pF(G)+\frac{||G_{\bar{z}}||^2}{||F(G)||^2}\<G,G\>\bar{G}_pF(G)+\frac{\<G_{\bar{z}},\overline{F(G)}\>}{||F(G)||^2}(F(G))_z
\end{equation}
 and 
 \begin{equation}\label{cond}
 G_{\bar{z}}\wedge F(G)=0,  
\end{equation}
 where $F(G)\neq0$ and $F\colon\mathbb{C}^p\to\mathbb{C}^p$ is given by  
$$
F(z)=\left(\frac{1+|\<z,z\>|^2}{2}+||z||^2\right)e_p-z_p\left(\<\bar{z},\bar{z}\>z+\bar{z}\right),\, \quad z=(z_1,\cdots, z_p)\in\mathbb{C}^p.
$$
Moreover, the immersion $g$ is given by

\begin{multline}\label{par}
g={\rm Re}\Big(\int\frac{\<\bar{G}_z,F(G)\>}{||F(G)||^2}(1-\<G,G\>)hdz, \mathsf{i}\int\frac{\<\bar{G}_z,F(G)\>}{||F(G)||^2} (1+\<G,G\>)hdz,\\ 2\int \frac{\<\bar{G}_z,F(G)\>}{||F(G)||^2} G_1hdz,\dots, 2\int \frac{\<\bar{G}_z,F(G)\>}{||F(G)||^2} G_{p-1}hdz,(n-2)h\Big),
\end{multline}
where 
$$
h=\frac{1}{n-2}\left(\frac{n}{2}\right)^\frac{2}{n}e^{\frac{2}{n-2}{\rm Re}\int\frac{\<\bar{G}_z,F(G)\>}{||F(G)^2||}G_pdz}.
$$

Conversely, consider a non-holomorphic map $G=(G_1,\dots, G_p)\colon U\to\mathbb{C}^p$, such that $G_p\neq0$ and $F(G)\neq0$ on a simply connected open subset $U\subset\mathbb{C}$, that satisfies (\ref{eq}) and (\ref{cond}). Then, the surface $g\colon L^2\to \mathbb{R}^{p+2},p\geq1,$ parametrized by (\ref{par}) is an ($n-2$)-singular minimal surface with density vector $e=(0,\dots, 0,1).$
\end{theorem}

\begin{proof}
At first we assume that $g\colon L^2\to \mathbb{R}^{p+2}_+,p\geq1,$ is an ($n-2$)-singular minimal surface with density vector $e$. 
To begin with, we claim that $F(G)\neq0$. Suppose to the contrary that $F(G)=0$. Then, from equation \eqref{Gz} we have that $G_{\bar{z}}=0$ and thus, by \eqref{esoteriko} it follows that $|\<G,G\>|=1$. Hence, our assumption $F(G)=0$ is equivalently written as
\begin{equation}\label{6}
(\|G\|^2+1)e_p=G_p\left(\<\bar{G},\bar{G}\>G+\bar{G}\right).
\end{equation}
 By taking the inner product with $e_p$, we find that
\begin{equation}\label{ep}
\|G\|^2+1=G_p^2\<\bar{G},\bar{G}\>+|G_p|^2.
\end{equation}
Observe that $G_p\neq0$. If otherwise, it follows from equation \eqref{wz} that $\omega$ is constant, which contradicts our assumption that $g$ has substantial codimension. It follows from equation \eqref{ep} that $G_p^2\<\bar{G},\bar{G}\>$ is real. Using that $|\<G,G\>|=1$, we observe that
$$
G_p\<\bar{G},\bar{G}\>=\varepsilon\bar{G}_p,
$$
where $\varepsilon=\pm1$.
Plugging the above into equation \eqref{ep} it follows that $\varepsilon=1$ and thus, equation \eqref{ep} is equivalently written as
$$
\|G\|^2+1=2|G_p|^2.
$$
Hence, equation \eqref{6} becomes
$$
2\bar{G}_pe_p=\<\bar{G},\bar{G}\>G+\bar{G}.
$$
Differentiating the above with respect to $z$, using that $G_{\bar{z}}=0$ and since  $|\<G,G\>|=1$ we conclude that $G_z=0$ and consequently $G$ is constant. Then, it follows from equation \eqref{p} that
$$
(\varphi_1,\cdots, \varphi_{p+2})=\Psi(b_1,\cdots, b_{p+2}),
$$
where $b_1,\cdots, b_{p+2}$ are constant complex numbers with $(b_1, b_2)\neq(0,0)$. Since, $\varphi_i=e^\frac{\eta}{2}y_{i,z}$, $i=1,2$, it follows that the function $b_2y_1-b_1y_2$ is constant. This contradicts that fact that $g$ has substantial codimension.

Theorem \ref{minimal} implies that the immersion $Y\colon L^2\to(\mathbb{R}^{p+1}\times\mathbb{R}_+,\widetilde{\<\cdot,\cdot\>})$ given by $Y=\tau\circ g$ is minimal and then, Lemma \ref{fi2} is applicable. Now, by differentiating \eqref{Gz} with respect to $z$ we obtain
\begin{equation}\label{Gzz}
G_{z\bar{z}}=\frac{1}{2}\eta''\omega_z\bar{\Psi}F(G)+\frac{1}{2}\eta'\bar{\Psi}_zF(G)+\frac{1}{2}\eta'\bar{\Psi}(F(G))_z.
\end{equation}
Moreover, it follows from \eqref{Gz} that
\begin{equation}\label{Y}
\eta'\bar{\Psi}=2\frac{\<G_{\bar{z}},\bar{F(G)}\>}{\|F(G)\|^2}.
\end{equation}
The above equation immediately implies that $G$ cannot be holomorphic.

Taking into account \eqref{Y}, equation \eqref{Gz} is equivalently written as
\begin{equation}\label{Gz'}
G_{\bar{z}}=\frac{\<G_{\bar{z}},\bar{F(G)}\>}{\|F(G)\|^2}F(G),
\end{equation}
which immediately yields equation \eqref{cond}.
Using equations \eqref{wz}, \eqref{Y} and \eqref{Gz'}, it follows that \eqref{Gzz} is equivalent to the desired equation \eqref{eq}.
 
Since $\eta'=2(n-2)/n\omega$, it follows from \eqref{Y} that
$$
\Psi=\frac{n\omega}{n-2}\frac{\<\bar{G}_z,F(G)\>}{\|F(G)\|^2}.
$$
Hence, equation \eqref{wz} is written as
$$
({\rm log}\omega)_z=\frac{n}{n-2}\frac{\<\bar{G}_z,F(G)\>}{\|F(G)\|^2}G_p
$$
and thus
\begin{equation}
\omega=e^{\frac{n}{n-2}{\rm Re}\int\frac{\<\bar{G}_z,F(G)\>}{||F(G)^2||}G_pdz}.
\end{equation}
We see that the immersion $Y$ is locally parametrized as follows
\begin{equation}\label{Y'}
Y={\rm Re}\left(\int\Lambda\left(1-\<G,G\>\right)dz,\mathsf{i}\int \Lambda\left(1+\<G,G\>\right)dz, 2\int\Lambda G_1dz,\dots, 2\int\Lambda G_{p-1}dz,\omega \right),
 \end{equation}
where 
$$
\Lambda=\frac{1}{n-2}\left(\frac{n}{2}\right)^\frac{2}{n}\omega^\frac{2}{n}\frac{\<\bar{G}_z,F(G)\>}{\|F(G)\|^2}.
$$
Then, parametrization \eqref{par} follows from \eqref{Y'} and the fact that $g=\tau^{-1}\circ Y$, where $\tau\colon\left(\mathbb{R}_+^{p+2}, \<\cdot,\cdot\>_{n-2}\right)\to\left(\mathbb{R}^{p+1}\times\mathbb{R}_+,\widetilde{\<\cdot,\cdot\>}\right)$ is the aforementioned isometry.

Conversely, consider a non-holomorphic map $G=(G_1,\dots, G_p)\colon U\to\mathbb{C}^p$ with $G_p\neq0$, on a simply connected open subset $U\subset\mathbb{C}$ that satisfies (\ref{eq}) and (\ref{cond}). 
We now define the 1-form
$$
\Omega=\frac{\<G_{\bar{z}},\overline{F(G)}\>}{2\|F(G)\|^2}\bar{G}_pd\bar{z}+\frac{\<\bar{G_{\bar{z}}},F(G)\>}{2\|F(G)\|^2}G_pdz.
$$
Using \eqref{eq} and \eqref{Gz'} it follows that 
$$
\left(\frac{\<G_{\bar{z}},\overline{F(G)}\>}{2\|F(G)\|^2}\bar{G}_p\right)_z=\|G_{\bar{z}}\|^2\|F(G)\|^2\left(\frac{1+|\<G,G\>|^2}{2}+ \|G\|^2-\frac{2(n-1)}{n-2}|G_p|^2\right).
$$
This means that $\Omega$ is a closed form.
Then, Poincare's Lemma implies that $\Omega$ is exact, namely there exists a real function $\nu$ that satisfies
$$
\Omega=\nu_{\bar{z}}d\bar{z}+\nu_zdz.
$$

We define the functions  
$$
\omega=e^{\frac{2n}{n-2}\nu} \quad \text{and} \quad \Psi=\frac{\omega_z}{G_p}.
$$
Using  \eqref{eq}, \eqref{Gz'}, \eqref{h} and the exactness of $\Omega$, we confirm that $\Psi$, $G$ and $\omega$ fulfil equations  \eqref{Psizbar}, \eqref{Gz} and \eqref{wz}.
We define the map $Y\colon U\subset\mathbb{C}\to(\mathbb{R}^{p+1}\times\mathbb{R}_+,\widetilde{\<\cdot,\cdot\>})$ given by \eqref{Y'}
and we argue on the open subset of $U$ where $Y$ is regular. Thus, $Y$ is minimal due to Lemma \ref{fi2}. Now, define the immersion $g=\tau^{-1}\circ Y$. Observe that $g$ is minimal in the Riemannian manifold $(\mathbb{R}_+^{p+2}, \<\cdot,\cdot\>_{n-2})$. Then, it follows from Theorem \ref{minimal} that $g$ is an $(n-2)$-singular minimal surface of $\mathbb{R}^{p+2}$ with density vector $e=(0,\dots,0,1)$. The fact that $g$ is parametrized as in \eqref{par} follows from the above parametrization of $Y$. \qed

\bigskip

\end{proof}

\bigskip 

\bigskip

{\bf Data Availibility} The author declare that the data supporting the findings of this study are available within the paper.
\bigskip

{\bf Declarations}

\bigskip

{\bf Conflict of interest} The author have no relevant financial or non-financial
interests to disclose.

\bigskip

\noindent Athina Eleni Kanellopoulou\\
University of Ioannina \\
Department of Mathematics\\
Ioannina--Greece\\
e-mail: alinakanellopoulou@gmail.com


\begin{thebibliography}{lll}


\bibitem{omori} L. J. Alias, P. Mastrolia and M. Rigoli, {\it Maximum Principles and Geometric Applications}. Springer Monographs in Mathematics. Springer, Cham, 2016.

\bibitem{DEB} D. E. Blair, {\it On a generalization of the catenoid}. Can. J. Math., {\bf XXVII} No. 2 (1975), 231-236. 

\bibitem{domes} R. Bohme, S. Hildebrandt and E. Tausch, {\it The two-dimensional analog of the catenary}. Pacific J. Math. {\bf 88}, (1980), 247-278.





\bibitem{DF1} M. Dajczer and L. Florit, {\it On Chen's basic equality}. Illinois J. Math. {\bf 42}, (1998), 97-106.









\bibitem{da} M. Dajczer and R. Tojeiro, {\it Submanifold Theory. Beyond an Introduction}. Universitext.
Springer, New York, 2019.



\bibitem{DRK2} U. Dierkes, {\it Curvature estimates for minimal hypersurfaces in singular spaces}. Invent. Math. {\bf 122} (1995), 453-473. 



\bibitem{DODA} M. do Carmo, M. Dajczer and F. Mercuri, {\it Compact conformally flat hypersurfaces}. Trans. Amer. Math. Soc. {\bf 288} (1985), 189--203.




















\bibitem{MM} A. Martinez and A. L. Martinez-Trivino, {\it A Weierstrass type representation for translating solitons and singular minimal surfaces}. J. Math. Anal. Appl. {\bf 516} (2022), 126528.










\end{thebibliography}
\end{document}